\documentclass[12pt]{article}
\usepackage{amsfonts}
\usepackage{amsthm}
\usepackage{amsmath}
\usepackage{enumerate}
\usepackage{color}
\usepackage{graphicx} 
\usepackage{caption}
\usepackage{subcaption}
\usepackage[font=small]{caption}

\usepackage{authblk}
\usepackage[titletoc,title]{appendix}

\def\inte{\int\limits}
\def\e{\mathsf{E}}

\newtheorem{thm}{Theorem}
\theoremstyle{definition}
\theoremstyle{definition}\newtheorem{lemma}{Lemma}

\author[$\dagger$]{Eugenio P. 
Balanzario\thanks{ebg@matmor.unam.mx}}
\affil[$\dagger$]{Centro de Ciencias Matem\'aticas,
Universidad Nacional Aut\'onoma de M\'exico,
Apartado Postal 61-3 (Xangari),
Morelia  Michoac\'an, M\'exico}

\begin{document}

\title{A note on the distribution of
clusters and deserts of prime numbers}

\maketitle

\begin{abstract}
In this note we consider the distribution
of values of weighted sums of the
von Mangoldt arithmetical function.
By using a formula for the distribution
of values of trigonometric polynomials,
we are able to present evidence 
supporting the claim that these
weighted sums follow a distribution with 
a normal-like behavior.
\end{abstract}

Explicit formulas.

\medskip

\noindent
{\bf MathSciNet classification}:
11K65, 11K70, 11N05, 11N37.

\section*{Introduction}
It is the purpose of this
note to consider the following 
sums involving the von Mangoldt
arithmetical function $\Lambda(j)$,
\begin{eqnarray}\label{sden}
S_\sigma(n)=\frac{1}{\sigma\sqrt{2\pi}}
\sum_{|j|\leq\eta\sigma}
\frac{n}{n+j}\Lambda(n+j)
\exp\Big\{-\frac{1}{2}\Big(\frac{j}{\sigma}
\Big)^2\Big\}.
\end{eqnarray}
This sum $S_\sigma(n)$ is a weighted count
of prime powers lying in an interval 
of the form $\Omega(n,\sigma)
=(n-\eta\kern.03cm\sigma,n+\eta
\kern.03cm\sigma)$ with $\eta$ very slowly 
increasing. From the Prime Number 
Theorem, one is to expect that $S_\sigma(n)$ is
approximately equal to 1 for all $n\geq n_0$
and $\sigma>\sigma_0$. 
Deviations of $S_\sigma(n)$
from its expected value 1 are interesting to
consider because they put into evidence the
existence of intervals $\Omega(n,\sigma)$ 
with a superabundance of
prime powers when $S_\sigma(n)>1$. In this
case we say that we have clusters of prime powers.
Moreover, whenever $S_\sigma(n)<1$, then we have 
an interval $\Omega(n,\sigma)$ with a 
relative deficiency of prime powers, and
in this case we have deserts of prime powers.

In this note, we apply a formula given by
M. Kac \cite{Kac} for the
distribution of values of trigonometric
polynomials, in order to determine
the distribution of values of 
$\widehat{S}_\sigma(x)$
as $x$ runs in an interval of the form $(a,b)$.
Here, $\widehat{S}_\sigma(x)$ is
an approximation for $S_\sigma(n)$ 
in formula (\ref{sden}) and is
given by the following explicit formula (\ref{dos}), 
presented in \cite{Balanzario}, save for
a minor simplification in the
expression for $\widehat{S}_\sigma(x)$, which
we will account for in the last 
section of this note.

\begin{thm}\label{t1}
Assume the Riemann hypotheses and
the simplicity of the zeros 
$\beta+i\gamma$ of the zeta function. 
Let $\sigma$, $\eta$ and $\theta$ be
positive real numbers. 
For $n\in\mathbb{N}$, let $S_\sigma(n)$ be
as in (\ref{sden}) and for $x\in\mathbb{R}$, let
\begin{align}\label{sdex}
\widehat{S}_\sigma(x) =&
1-
\frac{1}{\sqrt{n}}
\sum_{|\gamma|\leq n\theta/\sigma}
\exp\Big\{-\frac{1}{2}\Big(
\frac{\sigma\gamma}{n}\Big)^2\Big\}
\cos(\gamma\log x).
\end{align} 
Then we have that, as $n\to\infty$,
\begin{align}\label{dos}
S_\sigma(n)=\widehat{S}_\sigma(n)+O
\bigg(
\frac{\log n}{\eta e^{\frac{1}{2}
\eta^2}}+
\frac{n^{\frac{3}{2}}\log n}
{\theta e^{\frac{1}{2}
\theta^2}}\bigg).
\end{align}
\end{thm}

In the following theorem, $\widehat{N}_a^b(\lambda)$
is the number of $x\in(a,b)$ such that
$\widehat{S}_\sigma(x)=\lambda$. It turns out
that Kac's computation scheme does not allow us
to obtain an estimate for $\widehat{N}_a^b(\lambda)$.
Instead, following Kac's computation
scheme, we obtain an estimate for
the related quantity $\e\kern.03cm \widehat{N}_a^b(\lambda)$.
This quantity is obtained by writing $\cos(\gamma_j
\log(x) +\Phi_j)$ instead of $\cos(\gamma_j \log x)$
in formula (\ref{sdex}), where $\{\Phi_j\}$ is a set
of independent and uniformly distributed random
variables, and then computing an expected value
over these variables. Proceeding
in this manner, and after minor adjustments,
Kac's computation
scheme gives us the following result.

\begin{thm}\label{t2}
Assume the hypotheses of Theorem~\ref{t1}
concerning the Riemann zeta function.
Let $\sigma=\sqrt{\rho\kern.05cm b\log b}$
where $\rho>3$.
Let $\widehat{N}_a^b(\lambda)=\mathrm{Card}\{x\in(a,b):
\widehat{S}_\sigma(x)=\lambda\}$. 
Let $a=\vartheta\kern.03cm b$ where
$\vartheta\in(0,1)$. Then, for the expected value
of $\widehat{N}_a^b(\lambda)$, we have,
\begin{eqnarray}\label{integral}
\e\kern.03cm
\widehat{N}_a^b(\lambda)=
\frac{1+o(1)}{\sqrt{2}\kern.05cm
\pi\sigma}
\inte_a^b
\exp\Big\{-\frac{\sigma\sqrt{\pi}(\lambda-1)^2}
{\log\{t/(4\pi\sigma)\}}\Big\}\>dt
\end{eqnarray}
as $b\to\infty$, where $o(1)\to0$.
\end{thm}

Considering that we are more interested in $\widehat{N}_a^b(\lambda)$
rather than in $\e\kern.03cm \widehat{N}_a^b(\lambda)$,
the question arises as whether the right hand
side of formula (\ref{integral}) does provide
a relevant estimate for $\widehat{N}_a^b(\lambda)$.
In order to throw light on this question,
numerous computations where performed
with distinct values for $b$
and with $a=0.5\kern.03cm b$. 
The results of a selection of these
computations are illustrated in Figures
~\ref{fig1},~\ref{fig2} and~\ref{fig3}. 
In Figure~\ref{fig2}
we have set $b$ to be equal to
four distinct values up to $b=10^9$. 
Each one of the dots in
this Figure~\ref{fig2} is of 
the form $\big(\lambda,\widehat{N}_a^b(\lambda)\big)$
and in solid line it is plotted the right hand
side of formula (\ref{integral}) as
a function of $\lambda$. In Figure~\ref{fig2}
we appreciate a good agreement between the
observed values $\big(\lambda,\widehat{N}_a^b(\lambda)\big)$
and the normal-like distribution given by
Theorem~\ref{t2}.
These computations are evidence in support
to the claim that the right hand
side of formula (\ref{integral}) does indeed
approximately describe the behavior of $\widehat{N}_a^b(\lambda)$
as $\lambda$ runs over the interval $(a,b)$.
Thus, we have empirical evidence
supporting the claim that formula (\ref{integral})
does indeed describe the frequency 
of clusters and deserts of prime
numbers.

\begin{figure}
\begin{center}
\includegraphics[width=12.5cm]{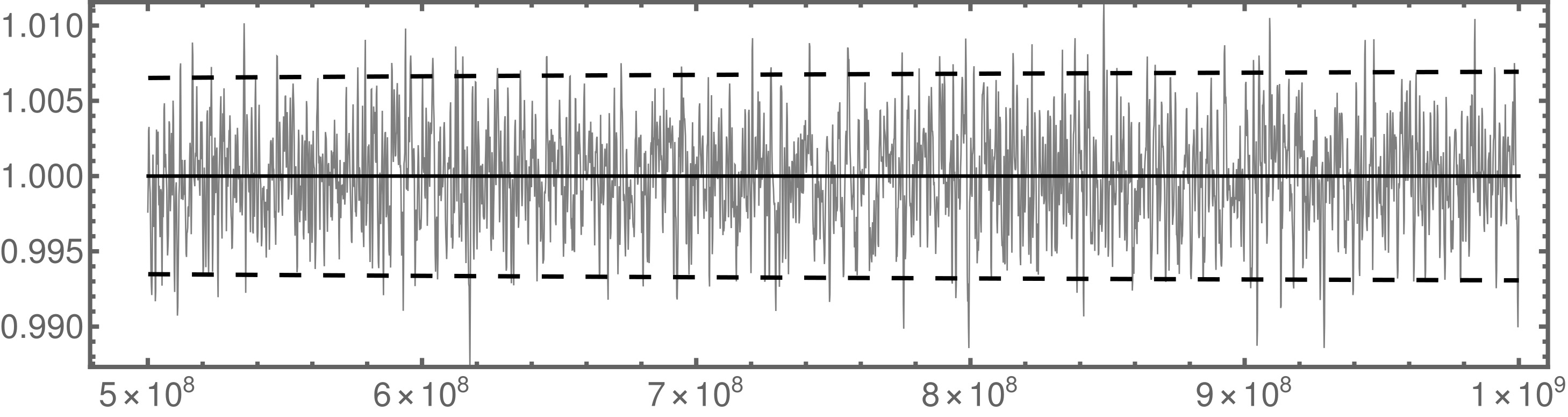}
\caption{Illustrates formula (\ref{sdex})
of Theorem~\ref{t1}
by showing $\widehat{S}_\sigma(x)$
for $x\in(a,b)$ with $b=10^9$ and
$a=0.5\kern.05cm b$. Notice that $\widehat{S}_\sigma(x)$
oscillates around its expected value 1. With
dashed lines are marked the levels $1\pm2\sigma$, 
so that approximately $95.45\%$ of the values
of $\widehat{S}_\sigma(x)$ lie within these 
limits.}
\label{fig1}
\end{center}
\end{figure}

\begin{figure}
\begin{center}
\includegraphics[width=12.5cm]{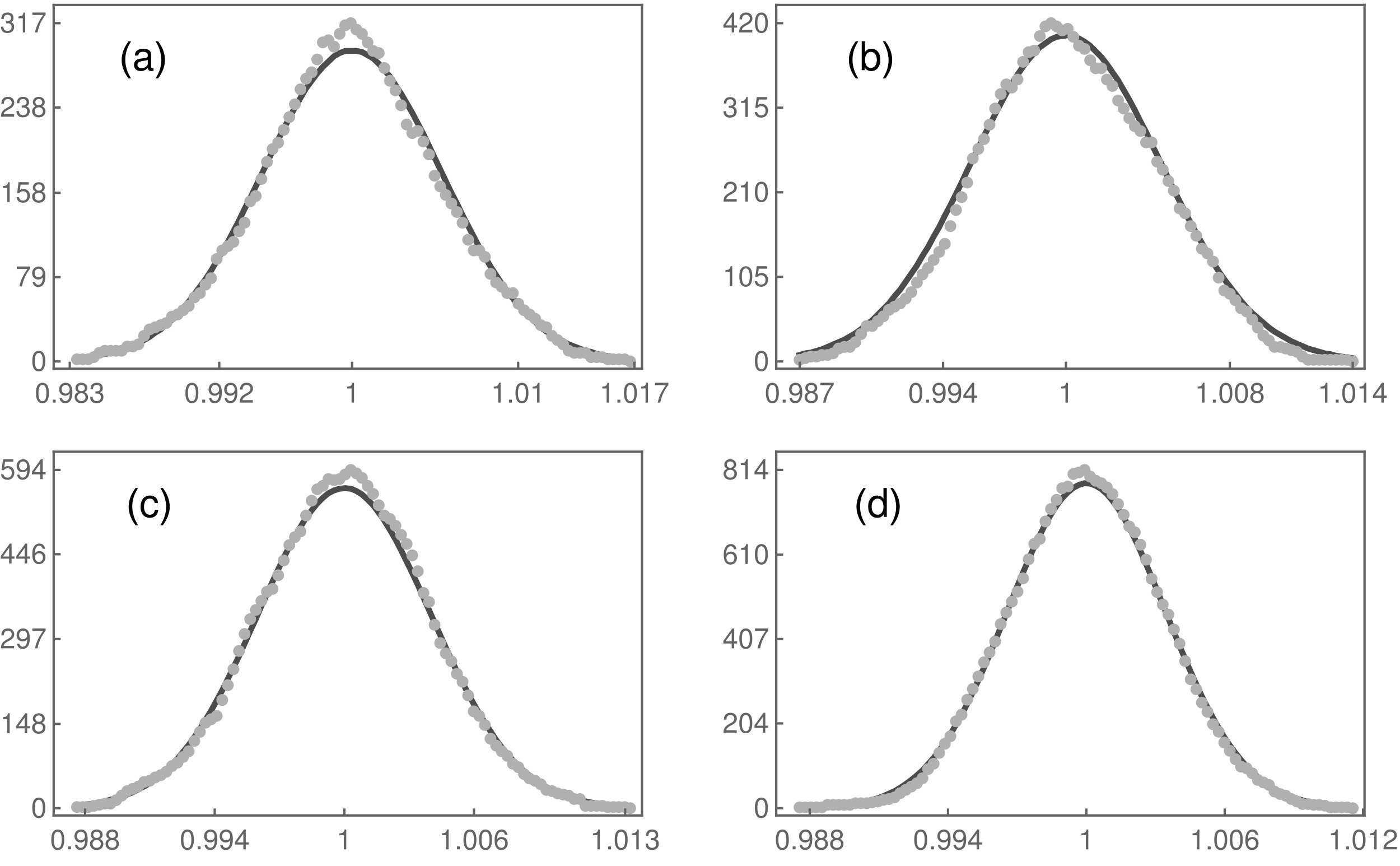}
\caption{Illustrates Theorem~\ref{t2} by
showing one hundred 
uniformly selected values of $\widehat{S}_\sigma(x)$
for each one of the following four values
of $b$. In part (a) of the figure we
have set $b=1.25\times10^8$, in part (b)
we have $b=2.5\times10^8$, in part (c)
we have $b=5\times10^8$ and in part (d)
we have $b=10^9$. In each part of this
figure, in continuous line
is the corresponding 
distribution of Theorem~\ref{t2}.}
\label{fig2}
\end{center}
\end{figure}

\begin{figure}
\begin{center}
\includegraphics[width=12.5cm]{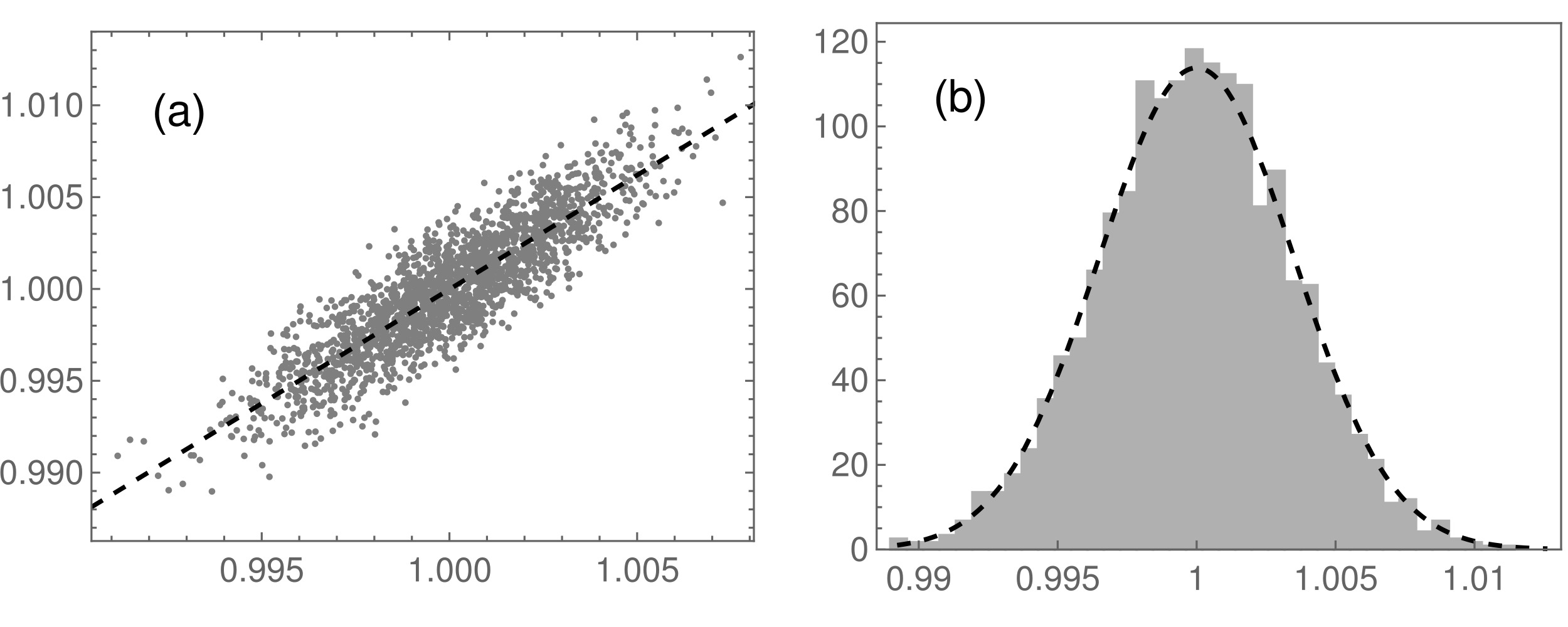}
\caption{In part (a) of the figure, we present the
dispersion diagram of 2\,005 points
of the form $\big(\widehat{S}_\sigma(x),
R_\sigma(x)\big)$ where $R_\sigma(x)$ is
as in equation (\ref{erre}). The values of
$\widehat{S}_\sigma(x)$ in the abscissa are those that
appear in Figure~\ref{fig1}
with $b=10^9$ and $a=0.5\kern.05cm b$.
In part (b) of the figure is
the histogram of the ordinates $R_\sigma(x)$ in
the previous dispersion diagram. In dashed line is
the density function of a normal probability
distribution with mean 1 and variance as
in formula (\ref{varianza}).}
\label{fig3}
\end{center}
\end{figure}

A word about the parameter $\eta$ in
formula (\ref{dos}) is in order (and
for the choice of the value of $\theta$,
see the proof of Theorem~\ref{t2} below). For
the purpose of approximating $S_\sigma(n)$
by $\widehat{S}_\sigma(n)$ for
$n\in(a,b)$, it is enough
to choose $\eta$ in such a way that
the first term within Landau's symbol
tends to zero, but for this to happen,
it is sufficient that $\eta$ grows
very slowly as $b\to\infty$. For example,
$\eta=2\sqrt{\log\log b}$ is an admissible 
choice for $\eta$. In all the computations
described in the preceding 
paragraph, a very close
agreement between $S_\sigma(n)$ and
$\widehat{S}_\sigma(x)$ was obtained with
$\eta=4$. Let $N_a^b(\lambda)=\mathrm{Card}
\{n\in(a,b):S_\sigma(n)=\lambda\}$. Hence,
provided that $\widehat{N}_a^b(\lambda)$, as
stated in Theorem~\ref{t2}, behaves as 
$\e\widehat{N}_a^b(\lambda)$, then $N_a^b(\lambda)$
will also have a distribution of values 
as given in the
right hand side of formula (\ref{integral}).

It is also interesting to address the question of how
much information does $\widehat{S}_\sigma(x)$
provide about the actual 
density of prime numbers in
an interval with center at $x$ and length $2\sigma$,
rather than about the weighted count of prime powers
given by $S_\sigma(n)$ in formula (\ref{sden}).
In order to throw light on this question,
a number of numerical experiments were performed.
For example, in part (a) of Figure~\ref{fig3} 
we show the dispersion
diagram of a set of 2\,005 
points of the form $\big(\widehat{S}_\sigma
(x), R_\sigma(x)\big)$ where $x\in(a,b)$
with $b=10^9$ and $a=0.5\kern.03cm b$ and
where we have set
\begin{align}\label{erre}
R_\sigma(x) =
\big(\pi(x+\sigma)-\pi(x-\sigma)\big)
\frac{\log x}{2\sigma},
\end{align}
$\pi(x)$ being the number of primes less
than or equal to $x$.
The computed slope of the fitted regression line
equals $m=1.26$ and the calculated
correlation coefficient 
is equal to $r=0.88$. In part (b)
of Figure~\ref{fig3} we show the
histogram of values $R_\sigma(x)$
in the previous dispersion diagram
and with a dashed line it is shown
the density function of a normal
probability distribution with mean
1 and variance equal to 
\begin{align}\label{varianza}
\Big(\frac{m}{r}\Big)^2\times
\frac{1}{2\kern.03cm\sigma\sqrt{\pi}}
\log\Big(\frac{b}{4\pi\sigma}\Big),
\end{align}
$m$ and $r$ being as above.
Notice that the second factor
in expression (\ref{varianza}) equals
$1/2s$, where $s$ is the coefficient of
$(\lambda-1)^2$ within the exponential
function in formula (\ref{integral}).
By repeating these computations
with distinct values for $b$ 
(but all the time with 
$a=0.5\kern.03cm b$), we always found
a very good agreement\footnote{As
measured by the $p$-value of
a statistical goodness of fit test.} 
between the
distribution of values of $R_\sigma(x)$
and the corresponding normal probability 
density function with mean 1 and
variance (\ref{varianza}). It was also 
found that the quotient $m/r$
is rather stable and oscillates
slightly around the number 1.42.

\section*{Proof of Theorem~\ref{t2}}

We begin the proof of Theorem~\ref{t2} by
determining an appropriate value 
for $\theta$ in formula (\ref{dos})
of Theorem~\ref{t1}. 
For this purpose, we require that 
\[
\frac{b^{\frac{3}{2}}\log b}
{\theta e^{\frac{1}{2}\theta^2}}
\leq\frac{\log b}{b^{\frac{1}{2}
(\rho-3)}}
\] 
with $\rho>3$, so that the
second term within Landau's symbol
in equation (\ref{dos})
tends to zero as $b\to\infty$.
For this purpose it is
enough if we set $\theta=
\sqrt{\rho\log b}$,
as stated in Theorem~\ref{t2}.

Now we notice that the cosine function in
equation (\ref{sdex}) is evaluated
in $\log t$ instead of in $t$, as it
is done in Kac's paper \cite{Kac}. Thus, 
a change of variable is in order.
Let $f(t)$ be such that
\[
\widehat{S}_\sigma(e^t)=1-2f(t),
\]
so that
\[
f(t)=\sum_{j=1}^k
a_j(t)\cos(\gamma_jt)
\kern.5cm\hbox{with}\kern.5cm
a_j(t)=\exp\Big\{-\frac{1}{2}
\Big(\Big(\frac{\sigma}{e^t}
\gamma_j\Big)^2+t\Big)\Big\},
\]
and $k=N(b\kern.03cm\theta/\sigma)$ where
$N(x)=\mathrm{Card}\{\gamma\in(0,x]\}$ 
counts the number of nontrivial
zeros $\beta+i\gamma$
of the Riemann zeta function with 
positive imaginary part less than $x$.

Notice that the coefficients $a_j(t)$ of
the trigonometric polynomial $f(t)$ are
functions of $t$. 
This fact is not accounted for
in \cite{Kac} and in what
follows we slightly modify Kac's computation
scheme so that it also applies in 
the present case, where $a_j(t)$ as
defined above are non constant.
Notice also that, for $\lambda\in\mathbb{R}$,
\begin{align}\label{cambio}
\widehat{S}_\sigma(t)=\lambda
\kern1cm\hbox{if and only if}\kern1cm
f(t)=\frac{1-\lambda}{2}.
\end{align}
Kac's computation scheme requires
that we consider the
following random trigonometric polynomial
\[
f_\lambda(t)=-\lambda+\sum_{j=1}^k
a_j(t)\cos(\gamma_jt+\Phi_j)
\]
where the random variables
$\Phi_1,\cdots,\Phi_k$
are independent and uniformly distributed
in the interval $(-\pi,\pi)$.
We are interested in counting the
number of values of $t\in(A,B)$ such that
$f_\lambda(t)=0$, where we
have set $A=\log a$ and $B=\log b$. For this end,
we partition the interval $(A,B)$
into $k'$ parts $(\beta_{j-1},\beta_j)$ such that
$f_\lambda'(t)$ does not change sign for 
$t\in(\beta_{j-1},\beta_j)$. The following 
double integral is to be considered,
\begin{eqnarray*}
I &=&\frac{1}{2\pi}
\inte_{-\infty}^{+\infty}
\bigg[\inte_A^B\cos\big(\xi\kern.03cm f_\lambda(t)\big)
|f_\lambda'(t)|\>dt\bigg]d\xi\\
\vbox{\kern.9cm}
&=&
\frac{1}{2\pi}\inte_{-\infty}^{+\infty}
\bigg[\sum_{j=1}^{k'} s_j
\inte_{\beta_j-1}^{\beta_j}
\cos\big(\xi\kern.03cm f_\lambda(t)
\big)f_\lambda'(t)\>dt\bigg]d\xi
\end{eqnarray*}
where $s_j=1$ if $f_\lambda'(t)>0$ for 
$t\in(\beta_{j-1},\beta_j)$,
and $s_j=-1$ if $f_\lambda'(t)<0$. Changing variables in 
the innermost  integral, we get 
\begin{eqnarray*}
I &=& \frac{1}{2\pi}
\inte_{-\infty}^{+\infty}
\bigg[\sum_{j=1}^{k'} s_j
\inte_{f_\lambda(\beta_j-1)}^{f_\lambda(\beta_j)}
\cos(\xi\kern.03cm t)\>dt\bigg]\>d\xi\\
\vbox{\kern.9cm}
&=&
\sum_{j=1}^{k'} \frac{1}{2\pi}\Bigg|
\inte_{-\infty}^{+\infty}
\frac{\sin\big(\xi\kern.03cm f_\lambda(\beta_j)\big)-
\sin\big(\xi\kern.03cm f_\lambda(\beta_{j-1})\big)}
{\xi}\>d\xi\Bigg|.
\end{eqnarray*}
This last expression is a sum
of Dirichlet integrals and therefore
is equal to the number
of zeros of $f_\lambda(t)$ in the interval
$(A,B)$. Since (\ref{cambio}) is true,
then our task now is to
compute $\e(I)=\e\kern.03cm \widehat{N}_a^b(1-2\lambda)$.

For the computation of $\e(I)$,
in what follows the integral $I$ 
is considered form
another point of view. For this purpose, 
Kac notices that, for $x\in\mathbb{R}$ we have,
\[
|x|=\frac{1}{\pi}
\inte_{-\infty}^{+\infty}
\frac{1-\cos(\eta\kern.03cm x)}{\eta^2}\>d\eta.
\]
Therefore,
\begin{gather}\label{ii}
I=\frac{1}{2\pi^2}
\inte_{-\infty}^{+\infty}\bigg[
\inte_A^B\inte_{-\infty}^{+\infty}
\cos\big(\xi\kern.03cm f_\lambda(t)\big)
\frac{1-\cos\big(\eta\kern.03cm f_\lambda'(t)\big)}
{\eta^2}\>d\eta
\>dt\bigg]d\xi.
\end{gather}
In order to  compute
the expected value of $I$,
use is made of the identity
\[
4\cos(\xi\kern.03cm f_\lambda)
\cos(\eta\kern.03cm f_\lambda')=\sum
\exp\{i(\pm \xi f_\lambda\pm\eta f_\lambda')\},
\]
where the sum is over the four
combinations of the signs inside
the exponential function. The expression
$\exp\{i(\pm\xi f_\lambda\pm\eta f_\lambda')\}$
is a function of the random
variables $\Phi_1,\cdots\Phi_k$
and is itself a random variable
with expected value
\begin{align*}
Q &= \frac{1}{(2\pi)^k}
\inte_{-\pi}^{+\pi}\cdots\inte_{-\pi}^{+\pi}
\exp\{i(\pm \xi\kern.03cm f_a
\pm\eta\kern.03cm f_a')\}
\>d\varphi_1\cdots d\varphi_k\\
&=
\cos(\xi\kern.03cm \lambda)
\prod_{j=1}^k
\frac{1}{2\pi}\inte_{-\pi}^{+\pi}
e^{i(\pm\xi \kern.03cm
a_j(t)\cos(\gamma_jt+\varphi_j)\mp
\eta\kern.03cm a_j(t)\gamma_j
\sin(\gamma_jt+\varphi_j))
}\>d\varphi_j.
\end{align*}
Given two real numbers $x$ and $y$,
a straightforward computation gives
\begin{gather*}
\frac{1}{2\pi}\inte_{-\pi}^{+\pi}
\exp\big\{i\big(\pm x\cos(\gamma t+\varphi)
\mp y\sin(\gamma t+\varphi)
\big)\big\}\>d\varphi=J_0\big(\sqrt{x^2+y^2}\big),
\end{gather*}
where $J_0(x)$ is the Bessel
function of the first kind.
Therefore,
\begin{eqnarray}\label{q}
Q=\cos(\xi\kern.03cm \lambda)
\prod_{j=1}^k J_0\Big(a_j(t)
\sqrt{\xi^2+\eta^2\gamma_j^2}\Big).
\end{eqnarray}
In section number 2 of \cite{Kac}, Kac
gives an argument to justify taking
the expectation inside the triple
integral (\ref{ii}). After taking expectation
with respect to $\Phi_1,\cdots,\Phi_k$, formula
(\ref{ii}) becomes
\[
I=\frac{1}{2\pi^2}
\inte_{-\infty}^{+\infty}\bigg[
\inte_A^B\inte_{-\infty}^{+\infty}
D(\xi,\eta,t)\>d\eta
\>dt\bigg]d\xi
\]
where
\begin{align}\label{de}
D(\xi,\eta,t)=
\cos(\xi\kern.03cm\lambda)
\frac{\prod_{j=1}^kJ_0(a_j(t)|\xi|)-
\prod_{j=1}^kJ_0\Big(a_j(t)\sqrt{\xi^2+\eta^2
\gamma_j^2}\Big)}{\eta^2}.
\end{align}

Our aim is now to give $Q$, as given
in equation (\ref{q}), a more
manageable form. For this purpose we
write
\begin{eqnarray}\label{ese}
h_1(t)=
\sum_{j=1}^k
a_j^2(t)
\kern1.2cm\hbox{and}\kern1.2cm
h_2(t)=
\sum_{j=1}^k
a_j^2(t)\gamma_j^2.
\end{eqnarray}
Note that
for all $0<t<B=\log b$ and $j\in\mathbb{N}$ we have
\begin{eqnarray}\label{cota}
0\leq a_j(t)\leq
\frac{1}{(2e)^{\frac{1}{4}}
\sqrt{\sigma\kern.05cm\gamma_j}}<
\frac{1}{(b\log b)^{\frac{1}{4}}},
\end{eqnarray}
since $\sigma=\sqrt{\rho\kern.05cm b\log b}$.
Let $M=(b\log b)^{\frac{1}{4}}
\times o(1)$,
where the symbol $o(1)$ stands for appropriate 
functions tending slowly to zero as $b\to\infty$.
If $|\xi|\leq M$, then
we have that $\xi^2a_j^2(t)<o(1)$, and
since $\log J_0(x)=-x^2/4+O(x^4)$ as $x\to0$,
then
\begin{align}
\begin{split}\label{ji}
\prod_{j=1}^k
J_0(a_j(t)|\xi|)&=
\prod_{j=1}^k\exp\Big\{
-\frac{1}{4}\Big(\xi^2a_j^2(t)+
O\big(\xi^4a_j^4(t)\big)\Big)\Big\}\\
&=
\exp\Big\{-\frac{1}{4}\xi^2h_1(t)
(1+o(1))\Big\}.
\end{split}
\end{align}
On the other hand, 
for all $0<t<B=\log b$ and $\gamma_j\leq b\theta/\sigma$,
inequality (\ref{cota}) implies that
\begin{gather*}
a_j(t)\gamma_j<
\frac{1}{\sqrt{\sigma\kern.03cm\gamma_j}}\gamma_j
\leq\frac{\sqrt{b\kern.03cm\theta}}{\sigma}
<\frac{1}{(\log b)^{\frac{1}{4}}}.
\end{gather*}
Now let 
$N=(\log b)^{\frac{1}{4}}\times o(1)$ and notice that, 
if $|\xi|\leq M$ and $|\eta|\leq N$, then
\[
a_j(t)\sqrt{\xi^2+\eta^2\gamma_j^2}
\leq 
a_j(t)|\xi|+a_j(t)\gamma_j|\eta|
=o(1).
\]
Therefore we can write
\[
\prod_{j=1}^k J_0\Big(a_j(t)
\sqrt{\xi^2+\eta^2\gamma_j^2}\Big)=
\exp\bigg\{-\frac{1}{4}
\sum_{j=1}^k
a_j^2(t)(\xi^2+\eta^2\gamma_j^2)
+E\bigg\}
\]
where
\[
E\ll\sum_{j=1}^k
a^4_j(t)(\xi^2+\eta^2\gamma_j^2)^2
\ll o(1)\times\big(\xi^2h_1(t)+
\eta^2h_2(t)\big).
\]
In view of the above,
now we have, for $|\xi|\leq M$
and $|\eta|\leq N$,
\begin{align}\label{exp}
\prod_{j=1}^k J_0\Big(a_j(t)
\sqrt{\xi^2+\eta^2\gamma_j^2}\Big)=
\exp\Big\{-\frac{1}{4}\big(\xi^2\tilde{h}_1(t)
+\eta^2\tilde{h}_2(t)\big)\Big\}
\end{align}
where we have written 
$\tilde{h}_j(t)=h_j(t)(1+o(1))$, for $j=1,2$.
Using expressions (\ref{ji}) and 
(\ref{exp}) in formula
(\ref{de}) for $I$, we get
\begin{align*}
D(\xi,\eta,t) &=
\frac{\cos(\xi\kern.03cm\lambda)}{\eta^2}
\bigg[e^{-\frac{1}{4}\xi^2\tilde{h}_1(t)}-
e^{-\frac{1}{4}(\xi^2\tilde{h}_1(t)
+\eta^2\tilde{h}_2(t))}\Big]
\\
&=\cos(\xi\kern.03cm\lambda)
\kern.03cm e^{-\frac{1}{4}\xi^2\tilde{h}_1(t)}
\frac{1-e^{-\frac{1}{4}
\eta^2\tilde{h}_2(t)}}{\eta^2}.
\end{align*}
Therefore,
\begin{align*}
\e(I) &=
\frac{1}{2\pi^2}
\inte_A^B\bigg[
\inte_{-\infty}^{+\infty}\cos(\xi\kern.03cm \lambda)
\kern.03cm e^{-\frac{1}{4}\xi^2\tilde{h}_1(t)}
\inte_{-\infty}^{+\infty}
\frac{1-e^{-\frac{1}{4}
\eta^2\tilde{h}_2(t)}}{\eta^2}
\>d\eta\>d\xi\bigg]dt
\\
\vbox{\kern.9cm}
&=\frac{1}{2\pi^2}\inte_A^B\bigg[
\inte_{-M}^{+M}\cos(\xi\kern.03cm \lambda)
\kern.03cm e^{-\frac{1}{4}\xi^2\tilde{h}_1(t)}
\inte_{-N}^{+N}
\frac{1-e^{-\frac{1}{4}
\eta^2\tilde{h}_2(t)}}{\eta^2}
\>d\eta\>d\xi\bigg]dt+o\big(\e(I)\big),
\end{align*}
as $b\to\infty$.
Taking the two inner integrals
over the whole real line, we get
\begin{align}\label{edei}
\begin{split}
\e(I) &=
\frac{1}{2\pi^{\frac{3}{2}}}\inte_A^B
\sqrt{\tilde{h}_2(t)}
\inte_{-\infty}^{+\infty}\cos(\xi\kern.03cm \lambda)
\kern.03cm e^{-\frac{1}{4}\xi^2
\tilde{h}_1(t)}\>d\xi\, dt
+o\big(\e(I)\big)\\
\vbox{\kern.9cm}
&=
\frac{1}{\pi}\inte_A^B
\sqrt{\frac{\tilde{h}_2(t)}{\tilde{h}_1(t)}}
\exp\Big\{-\frac{\lambda^2}
{\tilde{h}_1(t)}\Big\}\>dt+
o\big(\e(I)\big)
\end{split}
\end{align}
as $b\to\infty$.

In the next lemma, we determine the
asymptotic behavior of $h_1(t)$ and $h_2(t)$
defined in equation (\ref{ese}).

\begin{lemma}\label{l1}
Let $\vartheta\in(0,1)$.
For $t\in(\vartheta\kern.03cm b,b)$ and
$\sigma=\sqrt{\rho\kern.05cm b\log b}$,
\begin{eqnarray*}
h_1(\log t) &=&
\frac{1}{4\sigma\sqrt{\pi}}\Big(
\log\Big(\frac{t}{4\pi\sigma}\Big)-
\frac{\gamma}{2}\Big)
+O\Big(\frac{1}{b}\Big),\\
\vbox{\kern.8cm}
h_2(\log t) &=&
\frac{t^2}{8\sigma^3\sqrt{\pi}}\Big(
\log\Big(\frac{t}{4\pi\sigma}
\Big)+\frac{2-\gamma}{2}\Big)+
O\Big(\frac{1}{b}\Big),
\end{eqnarray*}
as $b\to\infty$, where $\gamma$ is Euler's constant.
\end{lemma}

\begin{proof}
Here we sketch the proof of the approximation
for $h_1(\log t)$. Note that,
with $N(x)$ being the number of the positive
imaginary parts of the nontrivial zeros
of $\zeta(s)$ not greater than $x$,
\begin{align}\label{ene}
\begin{split}
h_1(&\log t) =
\frac{1}{t}\inte_0^{\sqrt{b}}
\exp\Big\{-\Big(
\frac{\sigma}{t}x\Big)^2\Big\}\>
dN(x)\\
&=\frac{2\sigma^2}{t^3}
\inte_{0}^\infty
x\kern.03cm N(x)
\exp\Big\{-\Big(
\frac{\sigma}{t}x\Big)^2\Big\}\>
dx+O\Big(\exp\Big\{
-\frac{1}{2}\vartheta^2\log^2 b
\Big\}\Big).
\end{split}
\end{align}
Let us recall that
\begin{eqnarray*}
N(x)\sim\frac{x}{2\pi}
\log\Big(\frac{x}{2\pi}\Big)-
\frac{x}{2\pi}+\frac{7}{8},
\end{eqnarray*}
as $x\to\infty$ (see \cite{Edwards}, page 134).
Using this relation in expression
(\ref{ene}), we find that $h_1(\log t)$ is
as stated in the lemma.
\end{proof}

Now we can finish the proof of
Theorem~\ref{t2}.
From Lemma~\ref{l1} we have
\[
\sqrt{\frac{
\tilde{h}_2(t)}{
\tilde{h}_1(t)}}=
\frac{e^t}{\sqrt{2}\kern.03cm\sigma}
\big(1+o(1)\big)
\]
as $b\to\infty$.
Hence, the expression for
$\e(I)$ in (\ref{edei}) becomes
\begin{align*}
\e(I) &=
\frac{1+o(1)}
{\sqrt{2}\kern.03cm\pi\kern.03cm\sigma}
\inte_A^B e^t\exp\Big\{
-\frac{\lambda^2}{\tilde{h}_1(t)}\Big\}\>dt\\
\vbox{\kern.9cm}
&=
\frac{1+o(1)}
{\sqrt{2}\kern.03cm\pi\kern.03cm\sigma}
\inte_a^b \exp\Big\{
-\frac{\lambda^2}{h_1(\log t)}\Big\}\>dt.
\end{align*}
Using the expression for $h_1(\log t)$
of Lemma~\ref{l1}, we have
\begin{align*}
\e(I)=
\frac{1+o(1)}
{\sqrt{2}\kern.03cm\pi\kern.03cm\sigma}
\inte_a^b \exp\Big\{
-\frac{4\sigma\sqrt{\pi}\kern.03cm\lambda^2}
{\log\{t/(4\pi\sigma)\}
}\Big\}\>dt.
\end{align*}
Recalling the second equation in (\ref{cambio}),
now we must perform the change of variable
$\lambda\mapsto(1-\lambda)/2$ in this
expression for $\e(I)=\e\kern.03cm \widehat{N}_a^b(1-2\lambda)$.
This finishes the proof of
Theorem~\ref{t2}.\qed

\section*{Proof of Theorem~\ref{t1}}

For the proof of Theorem~\ref{t1}, we
let $\alpha=(\mu/\sigma)^2$.
Then we have (see \cite{Balanzario}, 
formula (1.4))
\begin{align}\label{final}
S_\sigma(n) =
1-
\frac{\sqrt{n}}{\sigma}
\sum_{|\gamma|\leq n\theta/\sigma}
\widetilde{w}_\alpha(\gamma)
\Big(\frac{\sigma^2}{n}\Big)^{i\gamma}+E
\end{align}
where $\widetilde{w}_\alpha(\gamma)=
\Gamma(\alpha-1/2+i\gamma)/\Gamma(\alpha)$
and $E$ stands for the error term in formula
(\ref{dos}). Now we only need to simplify
the terms within the sum in formula (\ref{final}).
For $y$ such that $|y|\ll\sqrt{\alpha}$, we have
(see \cite{Balanzario}, 
Lemma 4.3)
\[
|\widetilde{w}_\alpha(y)|=
\frac{1}{\sqrt{\alpha}}\exp\Big\{
-\frac{y^2}{2(\alpha-1/2)}
\Big\}\Big\{1
+O\Big(\frac{1}{\alpha}
\Big)\Big\}.
\]
Let $z=x+iy$ and $A=\Gamma(\alpha+z)/
\Gamma(\alpha)$.
By Stirling's formula
\[
\log\Gamma(\alpha)
=\alpha\log\alpha-
\alpha+\frac{1}{2}
\log\frac{2\pi}{\alpha}+
O\Big(\frac{1}{\alpha}\Big)
\]
we have that
(see \cite{Balanzario}, 
proof of Lemma 4.3)
\begin{align*}
\log A
&=
\log\Gamma(\alpha+z)-
\log\Gamma(\alpha)\\
&=
\vbox{\kern.7cm}
\Big(\alpha+z-\frac{1}{2}\Big)\log(\alpha+z)
-(\alpha+z)
-\Big(\alpha-\frac{1}{2}\Big)\log(\alpha)
+\alpha+O\Big(\frac{1}{\alpha}\Big)\\
\vbox{\kern.7cm}
&=\log\alpha^z+
\Big(\alpha+z-\frac{1}{2}\Big)
\log\Big(1+\frac{z}{\alpha}\Big)-z
+O\Big(\frac{1}{\alpha}\Big).
\end{align*}
The imaginary part of
$\log A$ is equal to
\begin{align*}
y\log\alpha &+(\alpha-1)
\arctan\frac{y}{\alpha-1/2}+y
\log\Big|1+\frac{-1/2+iy}{\alpha}\Big|-y
+O\Big(\frac{1}{\alpha}\Big)\\
\vbox{\kern.7cm}
&=y\log\alpha+(\alpha-1)
\arctan\frac{y}{\alpha-1/2}-y
+O\Big(\frac{1}{\alpha}\Big)\\
\vbox{\kern.7cm}
&=y\log\alpha+(\alpha-1)
\frac{y}{\alpha-1/2}-y\\
\vbox{\kern.7cm}
&=y\log\alpha+O\Big(\frac{1}{\alpha}\Big).
\end{align*}
Hence, we have that
\[
\widetilde{w}_\alpha(y)=
\frac{1}{\sqrt{\alpha}}\exp\Big\{
-\frac{y^2}{2\alpha}
\Big\}\alpha^{iy}\Big\{1
+O\Big(\frac{1}{\alpha}
\Big)\Big\}.
\]
Now we notice that
\begin{align*}
\frac{\sqrt{n}}{\sigma}
\frac{1}{\sqrt{\alpha}}
&\exp\Big\{-\frac{1}{2\alpha}y^2\Big\}
\bigg[\alpha^{iy}
\Big(\frac{\sigma^2}{n}\Big)^{iy}+
\alpha^{-iy}
\Big(\frac{\sigma^2}{n}\Big)^{-iy}
\bigg]\\
\vbox{\kern.8cm}
&=
\frac{1}{\sqrt{n}}
\exp\Big\{-\frac{1}{2\alpha}y^2\Big\}
[n^{iy}+n^{-iy}]\\
\vbox{\kern.8cm}
&=
\frac{2}{\sqrt{n}}
\exp\Big\{-\frac{1}{2\alpha}y^2\Big\}
\cos(y\log n).
\end{align*}
This finishes the proof of
Theorem~\ref{t1}.\qed

\end{document}